\documentclass[11pt,letterpaper,reqno]{amsart}

\usepackage{amssymb}
\usepackage{amsmath}
\usepackage{amsthm}
\usepackage{bbm}
\usepackage{enumerate}
\usepackage{doi}

\usepackage{bookmark}
\usepackage{hyperref}
\hypersetup{pdfstartview={FitH}}

\numberwithin{equation}{section}

\newtheorem{theorem}{Theorem}[section]
\newtheorem{lemma}[theorem]{Lemma}
\newtheorem{proposition}[theorem]{Proposition}
\newtheorem{corollary}[theorem]{Corollary}

\theoremstyle{definition}
\newtheorem{remark}[theorem]{Remark}

\numberwithin{equation}{section}

\begin{document}

\title{On several irrationality problems for Ahmes series}

\author{Vjekoslav Kova\v{c}}
\address{V.K., Department of Mathematics, Faculty of Science, University of Zagreb, Bijeni\v{c}ka cesta 30, 10000 Zagreb, Croatia}
\email{vjekovac@math.hr}

\author{Terence Tao}
\address{T.T., UCLA Department of Mathematics, Los Angeles, CA 90095-1555}
\email{tao@math.ucla.edu}




\begin{abstract}
Using basic tools of mathematical analysis and elementary probability theory we address several problems on the irrationality of series of distinct unit fractions, $\sum_k 1/a_k$. In particular, we study subseries of the Lambert series $\sum_k 1/(t^k-1)$ and two types of irrationality sequences $(a_k)$ introduced by Paul Erd\H{o}s and Ronald Graham.
Next, we address a question of Erd\H{o}s, who asked how rapidly a sequence of positive integers $(a_k)$ can grow if both series $\sum_k 1/a_k$ and $\sum_k 1/(a_k+1)$ have rational sums. Our construction of double exponentially growing sequences $(a_k)$ with this property generalizes to any number $d$ of series $\sum_k 1/(a_k+j)$, $j=0,1,2,\ldots,d-1$, and, in particular, also gives a positive answer to a question of Erd\H{o}s and Ernst Straus on the interior of the set of $d$-tuples of their sums.
Finally, we prove the existence of a sequence $(a_k)$ such that all well-defined sums $\sum_k 1/(a_k+t)$, $t\in\mathbb{Z}$, are rational numbers, giving a negative answer to a conjecture by Kenneth Stolarsky.
\end{abstract}

\maketitle

\tableofcontents


\section{Introduction}
A series of unit fractions
\begin{equation}\label{eq:Ahmessum}
\sum_{k=1}^{\infty} \frac{1}{a_k} \quad\text{for some positive integers } a_1<a_2<a_3<\cdots
\end{equation}
was named an \emph{Ahmes series}\footnote{Ahmes was an ancient Egyptian scribe who (re)wrote the Rhind Mathematical Papyrus around 1550 B.C. The document contains, among other things, tables for converting numerous fractions $p/q$ into sums of several distinct unit fractions $1/n$. It is unknown to us why Erd\H{o}s and Straus preferred the term ``Ahmes series,'' but a possible reason could be avoiding finitary connotations of the commonly used notions of ``unit fractions'' and ``Egyptian fractions'' \cite{Gra13,BE22}. Another reason could be paying homage to the earliest historically recorded writer of a mathematical text.} by Erd\H{o}s and Straus \cite{ES64}, but the term has since been seldom used and always in relation with rationality/irrationality problems \cite{TY02,HT04}.
A folklore result is that
\begin{equation}\label{eq:suffices1}
\lim_{k\to\infty}a_k^{1/2^k}=\infty
\end{equation}
is a sufficient condition guaranteeing that the sum \eqref{eq:Ahmessum} is an irrational number \cite{ES64,E75}.
(Also see \cite[Theorem 1]{E75} for a slightly stronger claim.)
Conversely, the shifted Sylvester sequence \cite[A129871]{OEIS},
\begin{equation}\label{eq:Sylvester}
s_1=2, \quad s_{k+1} = s_k^2 - s_k + 1 \text{ for } k\geqslant 1,
\end{equation}
has asymptotics $s_k \approx c_0^{2^k}$ for a particular constant
\begin{equation}\label{eq:constVardi}
c_0=1.2640847\ldots
\end{equation}
(see \cite[p.~109]{GKP89}, \cite{Var91}) and the sum of its reciprocals equals $1$. By shifting Sylvester's sequence further, one then immediately obtains sequences satisfying $a_k \approx C^{2^k}$ for arbitrarily large constants $C$, the reciprocals of which still sum to a rational number. We conclude that the irrationality condition \eqref{eq:suffices1} is sharp; this observation is borrowed from \cite[p.~2]{E75}.

Numerous further questions arise, relating the (ir)rationality of the Ahmes series \eqref{eq:Ahmessum} with the growth and/or particular structure of the corresponding sequence $(a_n)$.
Some of them were addressed in the existing literature, a sample of which is \cite{ES64,San84,Bad87,Duv01,TY02,HT04}.
Here we attempt several open problems posed in Erd\H{o}s and Graham's 1980 book on problems in combinatorial number theory \cite{EG80} and in the proceedings \cite{Erd88} of the \emph{Symposium on Transcendental Number Theory} held in Durham in 1986. Parts of these questions also appeared on Thomas Bloom's website \emph{Erd\H{o}s problems} \cite{EP}, where they attracted our attention.

The main purpose of this manuscript is to show how far one can get with purely elementary techniques and constructions. Namely, quite a lot can be said by only recalling that the rational numbers form a countable and dense subset of $\mathbb{R}$, so that they are simultaneously easy and difficult to avoid.
Another purpose of this paper is to provide a convenient reference for future works on this type of irrationality problems, which will then, quite likely, require more involved techniques from number theory.

\subsection{Notation}
We write $a_n\approx b_n$ if real sequences $(a_n)$ and $(b_n)$ are \emph{asymptotically equal}, i.e., $\lim_{n\to\infty} a_n/b_n=1$.
As usually, $\mathbb{N}$ is the set of positive integers $1,2,3,\ldots$.
The indicator function of a set $S$ is written as $\mathbbm{1}_S$, while the cardinality of a finite set $S$ is simply denoted as $|S|$.
The logarithm to base $b$ is written as $\log_b$, but our canonical choice will always be the \emph{binary logarithm} $\log_2$.
We write $\lfloor x\rfloor$ for the \emph{floor} of $x\in\mathbb{R}$, i.e., the greatest integer not exceeding $x$.
The \emph{$\ell^\infty$-norm} of a real tuple $x=(x_i)_i$ is defined as $|x|_{\infty} := \max_i |x_i|$.
For two subsets $A$ and $B$ of the same abelian group, we write $A+B$ for their \emph{sumset} $\{x+y : x\in A, y\in B\}$.
Similarly, when $A$ is a subset of a real vector space and $\lambda\in\mathbb{R}$, then $\lambda A$ stands for $\{\lambda x : x\in A\}$.
Finally, the sum over the empty range of indices is always understood to be the neutral element in the ambient group.


\section{Statements of the results}

\subsection{One-dimensional results}

\subsubsection{General observations on subseries sums}
Completely generally, the Ahmes series \eqref{eq:Ahmessum} can be rewritten as a subseries
\begin{equation}\label{eq:subseries}
\sum_{n\in A} \frac{1}{n}
\end{equation}
of the harmonic series $\sum_n 1/n$ determined by some infinite set $A\subseteq\mathbb{N}$ and equivalently also as
\[ \sum_{n=1}^{\infty} \frac{\epsilon_n}{n} \]
for some $(\epsilon_n)_{n=1}^{\infty}\in\{0,1\}^{\mathbb{N}}$ such that infinitely many $\epsilon_n$ are equal to $1$.
Interesting questions can be asked about those subseries sums when $A$ is further restricted to a particular subset of $\mathbb{N}$.

A general observation is that the number \eqref{eq:subseries} is almost surely irrational in a relatively strong probabilistic sense: if we randomize $A$ on some infinite set $B\subseteq\mathbb{N}$, this almost surely guarantees irrationality, no matter how sparse $B$ is.

\begin{proposition}\label{prop:randomized}
Let $A, B \subseteq \mathbb{N}$ be such that $B$ is infinite and $\sum_{n\in A\cup B} 1/n < \infty$. Let $B'$ be a random subset of $B$, i.e., the events $\{n\in B'\}$ are independent with probability $1/2$ for each $n \in B$. Let $A' := A \Delta B'$ be the symmetric difference of $A$ and $B'$. Then $\sum_{n\in A'} 1/n$ is an irrational number with probability $1$.
\end{proposition}

To be more precise, we identify subsets $B'$ of $B$ with points in the infinite Cartesian product $\{0,1\}^B$. This is, in turn, the sample space on which the probability measure $\mathbb{P}_B$ is defined as the product of infinitely many copies of the symmetric Bernoulli measure on $\{0,1\}$, i.e., the unbiased coin tossing measure given by
\[ \mathbb{P}_1(\{0\})=1/2=\mathbb{P}_1(\{1\}). \]
Proposition \ref{prop:randomized} then claims
\[ \mathbb{P}_B\Big(\Big\{B'\in\{0,1\}^B : \sum_{n\in A \Delta B'}\frac{1}{n}\not\in\mathbb{Q}\Big\}\Big) = 1. \]
More details can be found in Section \ref{sec:proofsgeneral}.

\smallskip
The collection of the sets $A$ for which \eqref{eq:subseries} is  an irrational number is also prevailing in a topological sense.

\begin{proposition}\label{prop:nowheredense}
Let $B\in\mathbb{N}$ be an infinite set such that $\sum_{n\in B} 1/n <\infty$. The set of subsets $A$ of $B$ for which \eqref{eq:subseries} is rational is of the first category (i.e., a countable union of nowhere dense sets) in the Cantor space $\{0,1\}^{B}$ of all subsets of $B$, with the usual product topology.
\end{proposition}

From Proposition \ref{prop:nowheredense} and the Baire category theorem applied to the compact Hausdorff space $\{0,1\}^B$ it follows that the set
\[ \Big\{ A\in\{0,1\}^B \,:\, \sum_{n\in A} \frac{1}{n} \in \mathbb{R}\setminus\mathbb{Q} \Big\} \]
is dense in $\{0,1\}^B$, even though this is also easy to verify directly.

These simple propositions will be shown in Section \ref{sec:proofsgeneral}.


\subsubsection{Lambert subseries}
In general, a \emph{Lambert series} is a series of functions in a variable $t$ that takes form
\[ \sum_{n=1}^{\infty} \frac{\epsilon_n}{t^n-1} \]
for a sequence of complex coefficients $(\epsilon_n)_{n=1}^{\infty}$.
We are only concerned with the cases for which $\epsilon_n\in\{0,1\}$, i.e., in subseries
\begin{equation}\label{eq:LambertS2}
\sum_{n\in A} \frac{1}{t^n-1},
\end{equation}
for an infinite $A\subseteq\mathbb{N}$, of the full series
\begin{equation}\label{eq:Lambert}
\sum_{n=1}^{\infty} \frac{1}{t^n-1}.
\end{equation}
Chowla \cite{Cho47} conjectured and Erd\H{o}s \cite{Erd48} proved that \eqref{eq:Lambert} is an irrational number for every integer $t\geqslant 2$. More generally, Borwein \cite{Bor91,Bor92} solved an open problem of Erd\H{o}s by showing that
\[ \sum_{n=1}^{\infty} \frac{1}{t^n+q} \]
is irrational whenever $t\geqslant 2$ is an integer and $q$ is a rational number different from $0$ and any of $-t^n$; an alternative proof appeared in \cite{AZ98}.

Erd\H{o}s mentioned \cite[p.~222]{E68}:
\begin{quote}
\emph{In fact I know no example of an infinite sequence [of positive integers] $n_1<n_2<\cdots$ and [an integer] $t\geqslant 2$ for which
\begin{equation}\label{eq:LambertS}
\sum_{i=1}^{\infty} \frac{1}{t^{n_i}-1}
\end{equation}
is rational, though it seems likely that this can happen.}\footnote{There is an obvious typo in \cite[Formula (2)]{E68}, the series \eqref{eq:LambertS} being mistakenly written as $\sum_{t=1}^{\infty}1/(t^n-1)$, but it is clear what was meant from the rest of the text.}
\end{quote}
Erd\H{o}s and Graham repeated the particular case $t=2$ of this speculation in \cite[p.~62]{EG80}:
\begin{quote}
\emph{Perhaps $\sum_{k=1}^{\infty} 1/(2^{n_k}-1)$ is irrational for any [positive integers] $n_1<n_2<\cdots$.}
\end{quote}
Erd\H{o}s also mentioned this special case in \cite[p.~105]{Erd88} and it also appeared as \cite[Problem \#257]{EP}.

For every fixed integer $t\geqslant2$ the sums \eqref{eq:LambertS2} obtained as $A$ varies over all subsets of $\mathbb{N}$ are mutually distinct and they form a Cantor set (in the topological sense). An easy verification of this claim is given in Remark \ref{rem:LambertCantor} in Section \ref{sec:Lambert}.
Its immediate consequence is that there can be only at most countably many counterexamples in which \eqref{eq:LambertS} is rational.

We do not know how to solve the original question implied by Erd\H{o}s. In fact, it seems difficult, as the subseries $\sum_{n\in A} 1/(2^n-1)$ already incorporate various delicate series for special cases of the set $A$.
For example, if $A=\mathbb{P}$ is the set of all primes, then the series becomes
\begin{equation}\label{eq:omegaseries}
\sum_{p\in\mathbb{P}} \frac{1}{2^p-1} = \sum_{p\in\mathbb{P}} \sum_{k=1}^\infty \frac{1}{2^{kp}} = \sum_{n=1}^\infty \frac{\omega(n)}{2^n},
\end{equation}
where $\omega(n)$ denotes the number of distinct prime factors of $n$. It is still open if the last series has a rational sum; this is another old problem of Erd\H{o}s \cite[Problem \#69]{EP}.
Recently, Pratt \cite{Pra24} gave a conditional proof of the irrationality of \eqref{eq:omegaseries} assuming a certain uniform version of the Hardy--Littlewood prime tuples conjecture.

We can still give a negative result when several series \eqref{eq:LambertS} are merged in the following theorem.

\begin{theorem}\label{thm:mergedLambert}
Suppose that $m$ is a positive integer, while $2\leqslant t_1<t_2<\cdots<t_m$ are integers such that
\begin{equation}\label{eq:merLcond}
\sum_{k=1}^{m} \frac{1}{t_k-1} > 1.
\end{equation}
Then there exist sets $A_1,A_2,\ldots,A_m\subseteq\mathbb{N}$ such that at least one of them is infinite and that
\begin{equation}\label{eq:merLsum}
\sum_{k=1}^{m} \sum_{n\in A_k} \frac{1}{t_k^n-1}
\end{equation}
is a rational number.
\end{theorem}

Note that we are allowed to have repetition of the terms $1/(t_k^n-1)$ for different pairs $(k,n)$, even though, in those situations, the result does not deal with the Ahmes series anymore.


\subsubsection{Irrationality sequences}
In the literature one can find three possible definitions of when a sequence of positive integers
\begin{equation}\label{eq:increasingseq}
a_1<a_2<a_3<\cdots
\end{equation}
should be called an \emph{irrationality sequence}.
All three of them are making precise the claim that all ``perturbations'' of the Ahmes series $\sum_n 1/a_n$ sum up to an irrational number.

The first such definition is due to Erd\H{o}s and Straus \cite[p.~2--3]{E75}.
We say that \eqref{eq:increasingseq} is a \emph{Type 1 irrationality sequence} if
\[ \sum_{n=1}^{\infty} \frac{1}{t_n a_n} \]
is an irrational number for every sequence of positive integers $(t_n)_{n=1}^{\infty}$.
It is a folklore that a Type 1 irrationality sequence $(a_n)$ satisfies $\lim_{n\to\infty}a_n^{1/n}=\infty$; see a brief argument in \cite[p.~63]{EG80}.
Erd\H{o}s initially called the above condition \emph{property P} in \cite{E75} and showed that it is satisfied by $a_n=2^{2^n}$.
This is an essentially sharp result, since Han\v{c}l \cite[Corollary 2]{Han91} showed that every Type 1 irrationality sequence needs to satisfy
\[ \limsup_{n\to\infty} \frac{\log_2 \log_2 a_n}{n} \geqslant 1; \]
also see \cite[Theorem 2]{Han91} and \cite[Corollary 3]{HSS08} for slightly stronger results in both directions.

\smallskip
In \cite[p.~63]{EG80} Erd\H{o}s and Graham gave the other two concurrent definitions of irrationality sequences.
Let us call a sequence \eqref{eq:increasingseq} a \emph{Type 2 irrationality sequence} if
\[ \sum_{n=1}^{\infty}\frac{1}{b_n} \]
is an irrational number for every sequence of positive integers $(b_n)_{n=1}^{\infty}$ such that $b_n\approx a_n$.
Erd\H{o}s and Graham wrote \cite[p.~63]{EG80}:
\begin{quote}
\emph{With this definition, we do not even know if $a_n=2^{2^n}$ is an irrationality sequence. Probably an irrationality sequence of this type must also satisfy $a_n^{1/n}\to\infty$.}
\end{quote}
Also see \cite[Problem \#263]{EP}.
Erd\H{o}s added \cite[p.~105]{Erd88}:
\begin{quote}
\emph{The trouble with this definition is that we do not know a [single] non-trivial irrationality sequence (\ldots).}
\end{quote}

We can prove the following general result, which handles sequences with less than double exponential growth that do not oscillate too much in size. It also just barely fails to address $a_n=2^{2^n}$.

\begin{theorem}\label{thm:type2}
Suppose that $(a_n)_{n=1}^{\infty}$ is a strictly increasing sequence of positive integers such that $\sum_n 1/a_n$ converges and
\begin{equation}\label{eq:t2cond}
\lim_{n\to\infty} \frac{a_{n+1}}{a_n^2} = 0.
\end{equation}
Then $(a_n)_{n=1}^{\infty}$ is not a Type 2 irrationality sequence.
\end{theorem}

In particular, every sequence $(a_n)_{n=1}^{\infty}$ satisfying $a_n\approx 2^{(2-\varepsilon)^n}$ for some $0<\varepsilon<1$ is not a Type 2 irrationality sequence.
In the positive direction, a sequence satisfying
\[ \liminf_{n\to\infty} \frac{a_{n+1}}{a_n^{2+\varepsilon}}>0 \]
for some $\varepsilon>0$ fulfills \eqref{eq:suffices1} and so it is a Type 2 irrationality sequence.
Also note that the sequence defined by $a_n = \lfloor c_0^{2^n} \rfloor$ for the constant \eqref{eq:constVardi} associated with Sylvester's sequence \eqref{eq:Sylvester} is clearly not a Type 2 irrationality sequence, as is seen by taking $b_n=s_n$.

\smallskip
Finally, we say that a sequence \eqref{eq:increasingseq} is a \emph{Type 3 irrationality sequence} if
\[ \sum_{n=1}^{\infty}\frac{1}{a_n + b_n} \]
is an irrational number for every bounded sequence of integers $(b_n)_{n=1}^{\infty}$ such that $b_n \neq 0$ and $a_n + b_n \neq 0$ for every index $n\in\mathbb{N}$.\footnote{The assumption $b_n\neq 0$ was not explicitly stated in \cite{EG80} or \cite{Erd88}. We included it in the definition to make meaningful the question on $a_n=2^n$ cited below. For most results of this paper this assumption makes no difference; see Remark \ref{rem:type3}. Alternatively, one could assume that $b_n>0$ and all of our proofs would remain the same.} Then Erd\H{o}s and Graham wrote \cite[p.~63]{EG80}:
\begin{quote}
\emph{In this case, $2^{2^n}$ is an irrationality sequence although we do not know about $2^n$ or $n!$.}
\end{quote}
Erd\H{o}s repeated the problem a few years later \cite[p.~105]{Erd88}, also adding:
\begin{quote}
\emph{Is there an irrationality sequence $a_n$ of this type which increases exponentially? It is not hard to show that it cannot increase slower than exponentially.}
\end{quote}
The question about $2^n$ and $n!$ recently also appeared on the website \emph{Erd\H{o}s problems} \cite[Problem \#264]{EP}.

We can prove the following.

\begin{theorem}\label{thm:type3}
Suppose that $(a_n)_{n=1}^{\infty}$ is a strictly increasing sequence of positive integers such that $\sum_n 1/a_n$ converges and
\begin{equation}\label{eq:t3cond}
\liminf_{n\to\infty} \biggl( a_n^2 \sum_{k=n+1}^{\infty} \frac{1}{a_k^2} \biggr) > 0.
\end{equation}
Then $(a_n)_{n=1}^{\infty}$ is not a Type 3 irrationality sequence.
\end{theorem}

Theorem \ref{thm:type3} gives a negative answer to the question of Erd\H{o}s and Graham about $a_n=2^n$, and the proof also generalizes immediately to every $a_n$ with a precisely exponential growth. In fact, we have the following immediate consequence.

\begin{corollary}\label{cor:type3}
If a strictly increasing sequence of positive integers  $(a_n)_{n=1}^{\infty}$ satisfies
\[ \limsup_{n\to\infty} \frac{a_{n+1}}{a_n} < \infty, \]
then it cannot be a Type 3 irrationality sequence.
\end{corollary}

In particular, a sequence $(a_n)_{n=1}^{\infty}$ satisfying $a_n\approx \theta^n$ for some $\theta\in(1,\infty)$ is not a Type 3 irrationality sequence.
That way Corollary \ref{cor:type3} generalizes a result by Han\v{c}l and Tijdeman \cite[Theorem 4.1]{HT04}, who constructed, for a given $\theta\in(1,\infty)\setminus\mathbb{Q}$, a sequence of positive integers \eqref{eq:increasingseq} such that $\lim_{n\to\infty}a_{n+1}/a_n=\theta$ and that $\sum_{n=1}^{\infty}1/a_n$ is a rational number.
Namely, we know by Corollary \ref{cor:type3} that $a_n=\lfloor\theta^n\rfloor$ (appropriately modified for small $n$ to become strictly increasing) is not a Type 3 irrationality sequence. Thus, there exists a bounded integer sequence $(b_n)_{n=1}^{\infty}$ such that $\sum_{n=1}^{\infty}1/(a_n+b_n)\in\mathbb{Q}$, but we still have $\lim_{n\to\infty}(a_{n+1}+b_{n+1})/(a_n+b_n)=\theta$.

In the positive direction we will give a probabilistic proof of the existence of Type 3 irrationality sequences that grow only a bit faster than exponentially.

\begin{theorem}\label{thm:type3faster}
If $F\colon\mathbb{N}\to(0,\infty)$ is a function such that
\begin{equation}\label{eq:type3fa0}
\lim_{n\to\infty} \frac{F(n+1)}{F(n)}=\infty,
\end{equation}
then there exists a Type 3 irrationality sequence $(a_n)_{n=1}^{\infty}$ such that $a_n \approx F(n)$.
\end{theorem}

In particular, there is a Type 3 irrationality sequence $(a_n)_{n=1}^{\infty}$ with asymptotics, say, $a_n \approx 2^{n \log_2 \log_2 \log_2 n}$.
In fact we will show a stronger result in Proposition \ref{prop:type3faster} in Section \ref{sec:type3construction}: in a certain sense, almost every sequence with a prescribed super-exponential growth \eqref{eq:type3fa0} will be a Type 3 irrationality sequence.

The variant of the above question of Erd\H{o}s and Graham for $a_n=n!$ is probably more difficult.
Namely, Theorem \ref{thm:type3faster} gives a Type 3 irrationality sequence that is asymptotically equal to $n!$, so the mere growth of the factorials cannot be enough to give the negative answer (as it did with $2^n$).
In fact, the more precise Proposition \ref{prop:type3faster} below yields a Type 3 irrationality sequence $(a_n)$ such that, say, $|a_n-n!| \leqslant \log_2 \log_2 n$.
In the positive direction, it is already an open problem whether
\[ \sum_{n=2}^{\infty} \frac{1}{n!-1} \]
is irrational; see \cite[p.~2]{E75} or \cite[Problem \#68]{EP}.


\subsection{Higher-dimensional results}

Elementary techniques of this paper can also solve more difficult problems when applied carefully in higher dimensions.

\subsubsection{Simultaneous rationality}
Erd\H{o}s came up with the following problem on two Ahmes series, which has been posed on several occasions in the 1980s \cite[p.~64]{EG80}, \cite[p.~334]{Erd85}, \cite[p.~104]{Erd88}, and recently also as \cite[Problem \#265]{EP}. We choose the formulation stated by Erd\H{o}s in \cite{Erd88}:
\begin{quote}
\emph{Once I asked: Assume that $\sum \frac{1}{n_k}$ and $\sum \frac{1}{n_k-1}$ are both rational. How fast can $n_k$ tend to infinity? I was (and am) sure that $n_k^{1/k}\to\infty$ is possible but $n_k^{1/2^k}$ must tend to $1$. Unfortunately almost nothing is known. David Cantor observed that
\begin{equation}\label{eq:DCantorex}
\sum_{k=3}^{\infty} \frac{1}{\binom{k}{2}} \quad\text{and}\quad \sum_{k=3}^{\infty} \frac{1}{\binom{k}{2}-1}
\end{equation}
are both rational and we do not know any sequence with this property which tends to infinity faster than polynomially.}
\end{quote}
In \cite[p.~334]{Erd85} Erd\H{o}s specifically mentioned the exponential growth as already being an interesting open problem:
\begin{quote}
\emph{(\ldots) and we could never decide if $n_k$ can increase exponentially or even faster.}
\end{quote}
On a different occasion \cite[p.~64]{EG80} a weaker property was also mentioned:
\begin{quote}
\emph{If $1$ is replaced by a larger constant then higher degree polynomials can be used. For example, if $p(x)=x^3+6x^2+5x$ then both $\sum_{n\geqslant 2}\frac{1}{p(n)}$ and $\sum_{n\geqslant 2}\frac{1}{p(n)-12}$ are rational (since both $p(n)$ and $p(n)-12$ completely split over the integers). Similar examples are known using polynomials with degrees as large as $10$ (see \cite{HW54}).}\footnote{There is a harmless typo in \cite{EG80}, the two series being $\sum_{n\geqslant 1}1/p(n)$ and $\sum_{n\geqslant 1}1/(p(n)+8)$, which do not factor as claimed. We deciphered the correct meaning as above.}
\end{quote}
Also note that studying rationality of the two series $\sum_k 1/n_k$ and $\sum_k 1/(n_k-d)$, for a fixed positive integer $d$, is indeed a simpler problem. Namely, it is solved by multiplying by $d$ any sequence $(n_k)$ that makes both $\sum_k 1/n_k$ and $\sum_k 1/(n_k-1)$ rational.

Our goal is to show that even a doubly exponential growth of $(n_k)$ is possible, i.e., one can have $\lim_{k\to\infty} n_k^{1/\beta^k} = \infty$ for some $\beta>1$. In particular, this confirms that $\lim_{k\to\infty} n_k^{1/k} = \infty$ can happen, just as Erd\H{o}s suspected. In principle, this also gives a definite answer on the largest possible \emph{type} of growth of $(n_k)$, since anything faster than doubly exponential is prohibited by condition \eqref{eq:suffices1} and by considering only a single series $\sum_k 1/n_k$. However, it remains unclear what the optimal exponent of the doubly exponential growth is. In particular, our result falls short of answering the part of the question on whether one can go beyond $\lim_{k\to\infty} n_k^{1/2^k} = 1$.

Moreover, we can even prove a generalization to several ``consecutive'' Ahmes series and our approach relies only on a topological property of the subseries sums.

\begin{theorem}\label{thm:higherD}
For every positive integer $d$ there exists a number $\beta>1$ such that the subset of $\mathbb{R}^d$ defined as
\begin{align*}
\biggl\{ \Big( \sum_{k=1}^{\infty}\frac{1}{a_k}, \sum_{k=1}^{\infty}\frac{1}{a_k+1}, \ldots, \sum_{k=1}^{\infty}\frac{1}{a_k+d-1} \Big) \,:\ & \text{$(a_k)_{k=1}^{\infty}$ is a strictly increasing}\\[-3mm]
& \text{sequence in $\mathbb{N}$ such that } \lim_{k\to\infty} a_k^{1/\beta^k} = \infty \biggr\}
\end{align*}
has a non-empty interior.
\end{theorem}

The proof will give a concrete possible value of $\beta$, namely any number satisfying \eqref{eq:thebeta} below.
By the density of $\mathbb{Q}^d$ in $\mathbb{R}^d$ we also have the following immediate consequence.

\begin{corollary}\label{cor:higherD}
For every $d\in\mathbb{N}$ there exist $\beta\in(1,\infty)$ and a strictly increasing sequence of positive integers $(a_k)_{k=1}^{\infty}$ such that
\[ \lim_{k\to\infty} a_k^{1/\beta^k} = \infty \]
and
\[ \sum_{k=1}^{\infty} \frac{1}{a_k+j} \in \mathbb{Q} \]
for $j=0,1,\ldots,d-1$.
\end{corollary}

For example, there exists a double exponentially growing sequence $(a_k)_{k=1}^{\infty}$ such that
\[ \sum_k \frac{1}{a_k}, \quad \sum_k \frac{1}{a_k+1}, \quad\text{and}\quad \sum_k \frac{1}{a_k+2} \]
sum up to rational numbers.
The main strength of our approach is that we do not attempt to give an explicit construction of one such sequence $(a_k)$, which is in a sharp contrast with slowly growing examples like \eqref{eq:DCantorex}.


\subsubsection{Non-empty interior of tuples of subsums}
The above theorem also solves another conjecture originating in the work of Erd\H{o}s and Straus.
Erd\H{o}s repeatedly placed it alongside with the aforementioned problems on the irrationality of series (\cite[\S7]{EG80}, \cite{Erd85}, \cite{Erd88}), even though it is not, strictly speaking, an irrationality question.
Namely, Erd\H{o}s reported \cite[p.~65]{EG80} that he and Straus showed that the set
\begin{align*}
\bigg\{ (x,y) = \Big( \sum_{k=1}^{\infty}\frac{1}{a_k}, \sum_{k=1}^{\infty}\frac{1}{1+a_k} \Big) \,:\ & \text{$(a_k)_{k=1}^{\infty}$ is a strictly increasing}\\[-3mm]
& \text{sequence in $\mathbb{N}$},\ \ \sum_{k=1}^{\infty}\frac{1}{a_k}<\infty \bigg\}
\end{align*}
contains a non-empty open set.
Their proof has never been published, which Erd\H{o}s explained several years later \cite[p.~105]{Erd88}:
\begin{quote}
\emph{We never published our proof since we did not work out the $r$-dimensional case.}
\end{quote}
To avoid possible confusion about what was meant to be a higher-dimensional generalization, Erd\H{o}s and Graham formulated the problem explicitly in three dimensions in their 1980 book \cite[p.~65]{EG80}:
\begin{quote}
\emph{Is the same true in three (or more) dimensions, e.g., taking all $(x,y,z)$ with
\begin{equation}\label{eq:3Dpoints}
x=\sum_k\frac{1}{a_k},\quad y=\sum_k\frac{1}{1+a_k},\quad z=\sum_k\frac{1}{2+a_k} ?
\end{equation}}
\end{quote}
Erd\H{o}s later mentioned this open question in a memorial article dedicated to Straus \cite{Erd85}, with an optimistic remark following the above two-dimensional claim \cite[p.~334]{Erd85}:
\begin{quote}
\emph{(\ldots) this no doubt generalizes for higher dimensions.}
\end{quote}
However, at the next opportunity \cite[p.~104--105]{Erd88} he posed the problem with a more reserved formulation:
\begin{quote}
\emph{Probably the analogous result holds for $r$ dimensions.}
\end{quote}

The three-dimensional case \eqref{eq:3Dpoints} of the question appeared as \cite[Problem \#268]{EP} and it was addressed recently, with the affirmative answer, by one of the present authors \cite{Kov24}. The proof in \cite{Kov24} was somewhat constructive, producing a concrete ball of radius $10^{-24}$ inside the set of points \eqref{eq:3Dpoints}.
Our Theorem \ref{thm:higherD} implies the above conjecture in all dimensions.

\begin{corollary}
For every positive integer $d$ the subset of $\mathbb{R}^d$ defined as
\[ \bigg\{ \Big(\sum_{n\in A}\frac{1}{n}, \sum_{n\in A}\frac{1}{n+1}, \ldots, \sum_{n\in A}\frac{1}{n+d-1}\Big) \,:\ A\subset\mathbb{N} \text{ infinite},\ \sum_{n\in A}\frac{1}{n}<\infty \bigg\} \]
has a non-empty interior.
\end{corollary}


\subsection{An infinite-dimensional result}
The ultimate result of this paper is a certain extension of Corollary \ref{cor:higherD} to infinitely many Ahmes series, which answers a question of Kenneth B. Stolarsky. Namely, Erd\H{o}s and Graham wrote in their 1980 problem book \cite[p.~64]{EG80}:
\begin{quote}
\emph{The following pretty conjecture is due to Stolarsky:
\[ \sum_{n=1}^{\infty}\frac{1}{a_n+t} \]
cannot be rational for every positive integer $t$.}
\end{quote}
Also see \cite[Problem \#266]{EP}.
The same problem appeared in \cite[p.~334]{Erd85} and \cite[p.~104]{Erd88}, with the exception that $t$ was also allowed to be an integer different from any of the numbers $-a_n$ for $n\in\mathbb{N}$.\footnote{In fact, the problem is stated equivalently there, in terms of the series $\sum_{n=1}^{\infty} 1/(a_n-t)$ and requiring $t\neq a_n$ for every $n$.}
We can disprove the conjecture, not only when $t$ ranges over the integers, but also, somewhat counterintuitively, when $t$ is allowed to be a rational number.

\begin{theorem}\label{thm:Stolarsky}
There exists a strictly increasing sequence of positive integers $(a_n)_{n=1}^{\infty}$ such that the series
\[ \sum_{n=1}^{\infty}\frac{1}{a_n+t} \]
converges to a rational number for every $t\in\mathbb{Q}\setminus\{-a_n:n\in\mathbb{N}\}$.
\end{theorem}

Our proofs of Theorems \ref{thm:higherD} and \ref{thm:Stolarsky} share common ingredients on the approximation of points in $\mathbb{R}^d$ by appropriate $d$-dimensional Ahmes subseries. The proof of Theorem \ref{thm:higherD} in Section \ref{sec:proofhigher} will then work in a fixed ambient dimension $d$, while the proof of Theorem \ref{thm:Stolarsky} in Section \ref{sec:proofStolarsky} will, in some sense, let $d\to\infty$.


\section{Proofs of Propositions \ref{prop:randomized} and \ref{prop:nowheredense} on general observations}
\label{sec:proofsgeneral}

We begin with a generally useful remark on series with positive terms.

\begin{remark}\label{rem:Kakeya}
Consider a real sequence $(x_n)_{n=1}^{\infty}$ with properties
{\allowdisplaybreaks
\begin{subequations}
\begin{align}
& x_n > 0 \text{ for all } n\in\mathbb{N}, \label{eq:KakCond1} \\
& x_1 \geqslant x_2 \geqslant x_3 \geqslant \cdots, \label{eq:KakCond2} \\
& \sum_{n=1}^{\infty} x_n < \infty. \label{eq:KakCond3}
\end{align}
\end{subequations}
}
It is an easy exercise, which has already been known to Kakeya \cite{Kak14a,Kak14b} (also see the survey paper \cite{BFP13}), that
\begin{equation}\label{eq:achievementset}
\biggl\{ \sum_{n=1}^{\infty} \epsilon_n x_n \,:\, (\epsilon_n)_{n=1}^{\infty}\in\{0,1\}^{\mathbb{N}} \biggr\}
\end{equation}
is a finite union of closed bounded intervals if and only if
\begin{equation}\label{eq:KakCond4}
\sum_{k=n+1}^{\infty} x_k \geqslant x_n \text{ for all sufficiently large } n\in\mathbb{N},
\end{equation}
while \eqref{eq:achievementset} is a Cantor set if (but not only if)
\begin{equation}\label{eq:KakCond5}
\sum_{k=n+1}^{\infty} x_k < x_n \text{ for all } n\in\mathbb{N}.
\end{equation}
Kakeya also conjectured that these two are the only possibilities for \eqref{eq:achievementset} up to a homeomorphism.
However, this was disproved by Weinstein and Shapiro \cite{WS80}, while Guthrie, Nymann, and S\'{a}enz \cite{GN88,NS00} gave a complete topological characterization of all possibilities for the set \eqref{eq:achievementset}.
Still, the above conditions are sufficiently simple and we will use them a few times in this paper.
Note that condition \eqref{eq:KakCond4} actually enables us to write every number from the interval union \eqref{eq:achievementset}, except for the left endpoints of the intervals forming it, as an infinite subseries of $\sum_{n=1}^{\infty}x_n$, i.e., in the form $\sum_{n=1}^{\infty} \epsilon_n x_n$ with infinitely many $1$s among the coefficients $\epsilon_n\in\{0,1\}$; see \cite[Problem~131, p.~29]{PS98}.\footnote{We are grateful to Yohei Tachiya for providing this reference and correcting an omission in the proof of Theorem \ref{thm:mergedLambert}.}
Also note that \eqref{eq:KakCond5} guarantees that every real number has at most one representation as a subseries sum from \eqref{eq:achievementset}.
\end{remark}

Now we turn to the proofs of our general propositions.

\begin{proof}[Proof of Proposition \ref{prop:randomized}]
We identify the subsets of $B$ with the points of $\{0,1\}^B$.
For every $T\subseteq B$ let $\mathbb{P}_T$ denote the (product) measure defined on the measurable space on $\{0,1\}^T$ generated by the cylinder sets
\[ \big\{ C\in\{0,1\}^T \,:\, \mathbbm{1}_{C}(t_1)=\epsilon_1,\, \ldots,\, \mathbbm{1}_{C}(t_m)=\epsilon_m \big\} \]
associated with every $m\in\mathbb{N}$, every $t_1<\cdots<t_m$ from $T$, and every $\epsilon_1,\ldots,\epsilon_m\in\{0,1\}$, such that
\[ \mathbb{P}_T \big(\big\{ C \,:\, \mathbbm{1}_{C}(t_1)=\epsilon_1,\, \ldots,\, \mathbbm{1}_{C}(t_m)=\epsilon_m \big\}\big) = \Bigl(\frac{1}{2}\Bigr)^m. \]
Choose a sequence $(b_n)_{n=1}^{\infty}$ of elements of $B$ such that $b_{n+1}>2b_n$ for every $n\in\mathbb{N}$, and denote $B_0:=\{b_n : n\in\mathbb{N}\}$.
The sequence $(x_n)_{n=1}^{\infty}$, $x_n:=1/b_n$ now clearly satisfies conditions \eqref{eq:KakCond1}--\eqref{eq:KakCond3} and \eqref{eq:KakCond5}. Consequently, every real number $x$, either cannot be represented as a subsum of $\sum_{n\in B_0}1/n$, or it has precisely one such representation, $x=\sum_{n\in S}1/n$ for some $S=\{s_1<s_2<\cdots\}\subseteq B_0$. In the latter case,
\begin{align*}
& \mathbb{P}_{B_0}\Big(\Big\{C : \sum_{n\in C} \frac{1}{n} = x\Big\}\Big)
= \mathbb{P}_{B_0}(\{S\}) \\
& = \lim_{m\to\infty} \mathbb{P}_{B_0}\big(\big\{C : \mathbbm{1}_{C}(b_1)=\mathbbm{1}_{S}(b_1),\,\ldots,\,\mathbbm{1}_{C}(b_m)=\mathbbm{1}_{S}(b_m)\big\}\big)
= \lim_{m\to\infty} \Bigl(\frac{1}{2}\Bigr)^m = 0,
\end{align*}
so in either case we have
\begin{equation}\label{eq:randomisx}
\mathbb{P}_{B_0}\Big(\Big\{C\in\{0,1\}^{B_0} \,:\, \sum_{n\in C} \frac{1}{n} = x\Big\}\Big) = 0.
\end{equation}
A general set $A'$ in question is obtained as $A'=C\cup D\cup (A\setminus B)$ for unique $C\subseteq B_0$ and $D\subseteq B\setminus B_0$.
Now, for every $q\in\mathbb{Q}$, the Tonelli--Fubini theorem gives
\begin{align*}
& \mathbb{P}_B \Big(\Big\{B'\in\{0,1\}^B \,:\, \sum_{n\in A'}\frac{1}{n}=q\Big\}\Big) \\
& = (\mathbb{P}_{B_0}\times\mathbb{P}_{B\setminus B_0}) \Big(\Big\{(C,D)\in\{0,1\}^{B_0}\times\{0,1\}^{B\setminus B_0} \,:\, \sum_{n\in C}\frac{1}{n} + \sum_{n\in D}\frac{1}{n} + \sum_{n\in A\setminus B}\frac{1}{n} = q\Big\}\Big) \\
& = \int \mathbb{P}_{B_0}\Big(\Big\{C\in\{0,1\}^{B_0} \,:\, \sum_{n\in C}\frac{1}{n} = q - \sum_{n\in D}\frac{1}{n} - \sum_{n\in A\setminus B}\frac{1}{n}\Big\}\Big) \,\textup{d}\mathbb{P}_{B\setminus B_0}(D) \stackrel{\eqref{eq:randomisx}}{=} 0.
\end{align*}
Since there are only countably many rational numbers, we conclude
\[ \mathbb{P}_B\Big(\Big\{B'\in\{0,1\}^B : \sum_{n\in A'}\frac{1}{n}\in\mathbb{Q}\Big\}\Big) = 0. \]
This proves that the opposite event $\{B'\subseteq B : \sum_{n\in A'}1/n \in \mathbb{R}\setminus\mathbb{Q}\}$ has $\mathbb{P}_B$-probability $1$.
\end{proof}

\begin{proof}[Proof of Proposition \ref{prop:nowheredense}]
Note that the function
\[ \Phi\colon\{0,1\}^B \to \mathbb{R}, \quad \Phi(A) := \sum_{n\in A} \frac{1}{n} \]
is continuous.
Namely, for any $A_0\in \{0,1\}^B$ and $\varepsilon>0$ we take a finite set $B_\varepsilon\subset B$ such that $\sum_{n\in B\setminus B_\varepsilon}1/n<\varepsilon$ and observe that the open set
\[ \{A\in\{0,1\}^B \,:\, \mathbbm{1}_A(t)=\mathbbm{1}_{A_0}(t) \text{ for every } t\in B_\varepsilon \} \]
is mapped via $\Phi$ into an $\varepsilon$-neighborhood of the point $\Phi(A_0)\in\mathbb{R}$.
From the continuity of $\Phi$ it follows that the set
\[ \Big\{ A\in\{0,1\}^B \,:\, \sum_{n\in A} \frac{1}{n} = q \Big\} = \Phi^{-1}(\{q\}) \]
is closed in $\{0,1\}^B$ for each rational number $q$. Moreover, $\Phi^{-1}(\{q\})$ has empty interior, since it cannot contain a set of the form
\[ \{A\in\{0,1\}^B \,:\, \mathbbm{1}_A(t)=\epsilon_t \text{ for every } t\in B_0 \} \]
for any finite $B_0\subset B$ and a $\{0,1\}$-collection $(\epsilon_t)_{t\in B_0}$, as this would mean that the sum $\sum_{n\in A}1/n$ is completely determined by the set $A\cap B_0$, which clearly is not the case.
We finally conclude that
\[ \Big\{ A\in\{0,1\}^B \,:\, \sum_{n\in A} \frac{1}{n} \in \mathbb{Q} \Big\} = \bigcup_{q\in\mathbb{Q}} \Phi^{-1}(\{q\}) \]
is of the first category.
\end{proof}


\section{Proof of Theorem \ref{thm:mergedLambert} on Lambert subseries}
\label{sec:Lambert}

\begin{remark}\label{rem:LambertCantor}
As we have mentioned in the introduction, for a fixed $t\geqslant2$, the subsums \eqref{eq:LambertS2} are mutually different for different $A\subseteq\mathbb{N}$.
This follows from
\[ \sum_{l=n+1}^{\infty} \frac{1}{t^l-1} < \frac{1}{t^n-1} \]
for every $n\in\mathbb{N}$, which is, in turn, easily seen as
\[ \sum_{l=n+1}^{\infty} \frac{t^n-1}{t^l-1} < \sum_{l=n+1}^{\infty} \frac{t^n}{t^l} = \frac{1}{t-1} \leqslant 1. \]
This verifies condition \eqref{eq:KakCond5} and Remark \ref{rem:Kakeya} applies. The same observation proves that the subsums \eqref{eq:LambertS2} form a Cantor set.
\end{remark}

\begin{proof}[Proof of Theorem \ref{thm:mergedLambert}]
It is enough to show that the set
\begin{equation}\label{eq:bigmixedsum} 
\biggl\{  \sum_{k=1}^{m} \sum_{n\in A_k} \frac{1}{t_k^n-1} \,:\, A_1,A_2,\ldots,A_m\subseteq\mathbb{N},\ A_1\cup A_2\cup \cdots\cup A_m\text{ is infinite} \biggr\}, 
\end{equation}
composed of the infinite sums in \eqref{eq:merLsum}, contains a non-degenerate interval.
Let us order the multiset
\[ \Bigl\{ \frac{1}{t_k^n-1} \,:\, n\in\mathbb{N},\ k\in\{1,2,\ldots,m\}  \Bigr\} \]
into a decreasing sequence $(x_n)_{n=1}^{\infty}$.
In our case \eqref{eq:KakCond1}--\eqref{eq:KakCond3} are obvious by the construction. 
Thus, it remains to verify \eqref{eq:KakCond4} if we want to show that \eqref{eq:achievementset} is a finite union of bounded closed intervals and that, consequently, \eqref{eq:bigmixedsum} is a finite union of bounded half-open intervals.

Take $\varepsilon>0$ sufficiently small that condition \eqref{eq:merLcond} still gives
\begin{equation}\label{eq:Lambaux1}
(1-\varepsilon)\sum_{k=1}^{m} \frac{1}{t_k-1} > 1.
\end{equation}
Denote $N_0 = \lfloor \log_2 (1/\varepsilon) \rfloor +1$.
Fix an index $n\in\mathbb{N}$ sufficiently large that $n\geqslant N_0$ and $x_n \leqslant \min_{1\leqslant k\leqslant m} 1/(t_k^{N_0}-1)$. We can verify \eqref{eq:KakCond4} for such indices $n$.
Indeed, for each $k\in\{1,2,\ldots,m\}$ let $N_k\geqslant N_0$ be the largest positive integer such that $1/(t_k^{N_k}-1) \geqslant x_n$; then
\begin{equation}\label{eq:Lambaux2}
\frac{1}{t_k^{N_k}-1} \geqslant x_n > \frac{1}{t_k^{N_k+1}-1}
\end{equation}
for every $1\leqslant k\leqslant m$.
Since the sequence is decreasing, all of the fractions $1/(t_k^{j}-1)$ with $j\geqslant N_k+1$ are certainly enumerated by $(x_l)_{l=n+1}^{\infty}$.
Therefore,
\[ \sum_{l=n+1}^{\infty} x_l
\geqslant \sum_{k=1}^{m} \sum_{j=N_k+1}^{\infty} \frac{1}{t_k^j-1}
> \sum_{k=1}^{m} \sum_{j=N_k+1}^{\infty} \frac{1}{t_k^j} \\
= \sum_{k=1}^{m} \frac{1}{t_k^{N_k}} \frac{1}{t_k-1}. \]
Since the choice of $N_0$ and the fact $N_k\geqslant N_0$ guarantee $t_k^{N_k} \geqslant 2^{N_0} > 1/\varepsilon$, so that $t_k^{N_k} - 1 > (1-\varepsilon) t_k^{N_k}$, the last expression is greater than
\[ (1-\varepsilon) \sum_{k=1}^{m} \frac{1}{t_k^{N_k}-1} \frac{1}{t_k-1}
\stackrel{\eqref{eq:Lambaux2}}{\geqslant} x_n (1-\varepsilon) \sum_{k=1}^{m} \frac{1}{t_k-1}
\stackrel{\eqref{eq:Lambaux1}}{>} x_n, \]
as desired.
\end{proof}


\section{Proofs of Theorems \ref{thm:type2} and \ref{thm:type3} on irrationality sequences}

Both Theorem \ref{thm:type2} and Theorem \ref{thm:type3} will be consequences of the following simple lemma.

\begin{lemma}\label{lm:1D}
Let $(J_n)_{n=1}^{\infty}$ be a sequence of non-empty intervals of positive integers.
If
\[ \sum_{n=1}^{\infty} \frac{1}{\min J_n} < \infty \]
and
\begin{equation}\label{eq:lJ1Dassume}
\sum_{k=n+1}^{\infty} \frac{|J_k|-1}{(\min J_k)(\max J_k)} \geqslant \frac{1}{(\min J_n)(\min J_n+1)}
\end{equation}
for all sufficiently large indices $n$, then the set of infinite sums
\begin{equation}\label{eq:theperturbed1D}
\biggl\{ \sum_{n=1}^{\infty}\frac{1}{x_n} \,:\, x_n\in J_n \text{ for every } n\in\mathbb{N} \biggr\}
\end{equation}
has a non-empty interior in $\mathbb{R}$.
In particular, there exist numbers $x_n\in J_n$, $n\in\mathbb{N}$, such that $\sum_{n=1}^{\infty} 1/x_n \in\mathbb{Q}$.
\end{lemma}

If the intervals $(J_n)_{n=1}^{\infty}$ are increasing and non-overlapping, i.e.,
\[ \min J_1 \leqslant \max J_1 < \min J_2 \leqslant \max J_2 < \min J_3 \leqslant \max J_3 < \cdots, \]
then every choice $x_n\in J_n$ for each $n\in\mathbb{N}$ yields one Ahmes series $\sum_n 1/x_n$.
Otherwise, repetitions of the terms in $\sum_n 1/x_n$ are possible, which will be fine for the intended applications of Lemma \ref{lm:1D}.

\begin{proof}[Proof of Lemma \ref{lm:1D}]
Let us denote
\begin{equation}\label{eq:Iintervals}
I_n := \biggl[ \sum_{k=n}^{\infty} \frac{1}{\max J_k} , \sum_{k=n}^{\infty} \frac{1}{\min J_k} \biggr] \subset (0,\infty)
\end{equation}
for every $n\in\mathbb{N}$.
We first claim that the set equalities
\begin{equation}\label{eq:newincl}
I_n = I_{n+1} + \biggl\{ \frac{1}{j} \,:\, j\in J_{n} \biggr\}
\end{equation}
hold for all sufficiently large positive integers $n$.
Namely, the right hand side of \eqref{eq:newincl} is clearly contained in the left hand side, so it remains to show the opposite inclusion.
The closed interval $I_{n+1}$ has length
\[ \sum_{k=n+1}^{\infty} \Big( \frac{1}{\min J_k} - \frac{1}{\max J_k} \Big)
= \sum_{k=n+1}^{\infty} \frac{|J_k|-1}{(\min J_k)(\max J_k)}, \]
while consecutive reciprocals $1/j$ of $j\in J_{n}$, appearing in the finite set on the right hand side of \eqref{eq:newincl}, are mutually separated by at most
\[ \frac{1}{\min J_{n}} - \frac{1}{\min J_{n}+1} = \frac{1}{(\min J_{n})(\min J_{n}+1)}. \]
Assumption \eqref{eq:lJ1Dassume} applies for large $n$ and guarantees that the translates $I_{n+1}+1/j$, $j\in J_{n}$, of $I_{n+1}$ fully cover the whole segment $I_n$.

Let $m\in\mathbb{N}$ be such that \eqref{eq:newincl} is valid for every positive integer $n\geqslant m$.
Now we claim that the set of series' tail sums
\[ \biggl\{ \sum_{k=m}^{\infty}\frac{1}{x_k} \,:\, x_k\in J_k \text{ for every } k\geqslant m \biggr\} \]
is precisely equal to the whole interval $I_m$.
Namely, fix some $x\in I_m$. Thanks to \eqref{eq:newincl} we can inductively construct a sequence of numbers $x_k\in J_k$ for $k=m,m+1,m+2,\ldots$ such that
\[ x \in \sum_{k=m}^{n-1}\frac{1}{x_k} + I_{n} \]
for every integer $n\geqslant m$. Since lengths of real line segments \eqref{eq:Iintervals} converge to $0$, we conclude
$x=\sum_{k=m}^{\infty}1/x_k$.

As a consequence of the claim that we just proved, the set \eqref{eq:theperturbed1D} contains the closed interval
\[ \sum_{n=1}^{m-1}\frac{1}{\min J_n} + I_m. \]
This interval is clearly non-degenerate, as otherwise we would have $\min J_k=\max J_k$ and $|J_k|=1$ for all sufficiently large indices $k$, which violates condition \eqref{eq:lJ1Dassume}. This completes the proof.
\end{proof}

Somewhat similar constructions for different irrationality problems have already appeared in the literature \cite{Han91,TY02,HT04}.

\begin{proof}[Proof of Theorem \ref{thm:type2}]
Let $(a_n)$ be a strictly increasing sequence of positive integers with $\sum_n 1/a_n$ converging and \eqref{eq:t2cond} holding. Our goal is to locate a sequence $b_n \approx a_n$ of positive integers such that $\sum_n 1/b_n$ is rational.

We begin by constructing a sequence of positive integers $(c_n)_{n=1}^{\infty}$ such that
\begin{equation}\label{eq:type2aux1}
\lim_{n\to\infty}\frac{c_n}{a_n} = 0,
\end{equation}
and
\begin{equation}\label{eq:type2aux2}
\lim_{n\to\infty}\frac{a_{n-1}^2 c_n}{a_n^2} = \infty.
\end{equation}
One possible choice is to set $c_1:=1$ and
\[ c_n := \Big\lfloor \frac{a_n^{3/2}}{a_{n-1}} \Big\rfloor+1 \quad\text{for } n\geqslant 2. \]
Namely, note that because of
\begin{equation}\label{eq:type2aux3}
\lim_{n\to\infty}a_n=\infty
\end{equation}
we have
\[ \lim_{n\to\infty}\frac{c_n}{a_n}
\stackrel{\eqref{eq:type2aux3}}{=} \lim_{n\to\infty} \frac{1}{a_n} \frac{a_n^{3/2}}{a_{n-1}}
= \lim_{n\to\infty}\frac{a_n^{1/2}}{a_{n-1}}
= \Big(\lim_{n\to\infty}\frac{a_{n+1}}{a_{n}^2}\Big)^{1/2} \stackrel{\eqref{eq:t2cond}}{=} 0 \]
and
\[ \liminf_{n\to\infty}\frac{a_{n-1}^2 c_n}{a_n^2}
\geqslant \liminf_{n\to\infty}\frac{a_{n-1}^2}{a_n^2} \frac{a_n^{3/2}}{a_{n-1}}
= \lim_{n\to\infty}\frac{a_{n-1}}{a_n^{1/2}}
= \Big(\lim_{n\to\infty}\frac{a_{n}^2}{a_{n+1}}\Big)^{1/2} \stackrel{\eqref{eq:t2cond}}{=} \infty. \]

Now we apply Lemma \ref{lm:1D} to the integer intervals $(J_n)_{n=1}^{\infty}$ defined as
\[ J_n := \big\{ a_n, a_n-1, a_n-2, \ldots, \max\{a_n-c_n, 1\} \big\}. \]
From \eqref{eq:type2aux1} and \eqref{eq:type2aux3} we know that
\[ \min J_n = a_n-c_n \quad\text{and}\quad |J_n|=c_n+1 \]
for every sufficiently large index $n$.
Thus, $\sum_n 1/\min J_n$ converges if and only if $\sum_n 1/(a_n-c_n)$ converges, but the last series can be compared to a convergent series $\sum_n 1/a_n$ thanks to \eqref{eq:type2aux1}.
Condition \eqref{eq:lJ1Dassume} of Lemma \ref{lm:1D} is also satisfied for sufficiently large $n$:
\begin{align*}
& (\min J_n)(\min J_n+1) \sum_{k=n+1}^{\infty} \frac{|J_k|-1}{(\min J_k)(\max J_k)} \\
& > (a_n-c_n)^2 \frac{c_{n+1}}{(a_{n+1}-c_{n+1})a_{n+1}}
= \Big(1 - \frac{c_n}{a_n}\Big)^2 \Big(1 - \frac{c_{n+1}}{a_{n+1}}\Big)^{-1} \frac{a_{n}^2 c_{n+1}}{a_{n+1}^2} .
\end{align*}
Namely, by \eqref{eq:type2aux1} and \eqref{eq:type2aux2} the last expression converges to $\infty$ as $n\to\infty$, so it becomes greater than $1$ whenever $n$ is large enough.
Lemma \ref{lm:1D} then gives a sequence of integers $(b_n)_{n=1}^{\infty}$ such that $\sum_{n=1}^{\infty} 1/b_n$ is a rational number and $a_n-c_n\leqslant b_n\leqslant a_n$ for each $n\in\mathbb{N}$, which implies $\lim_{n\to\infty} b_n/a_n=1$ by  the sandwich theorem and \eqref{eq:type2aux1} again.
\end{proof}

\begin{proof}[Proof of Theorem \ref{thm:type3}]
Let $C\in\mathbb{N}$ be large enough that
\begin{equation}\label{eq:type3aux1}
\frac{1}{a_n^2} \leqslant C \sum_{k=n+1}^{\infty} \frac{1}{a_k^2}
\end{equation}
holds for every $n\in\mathbb{N}$; it exists by condition \eqref{eq:t3cond}.
We apply Lemma \ref{lm:1D} to the integer intervals $(J_n)_{n=1}^{\infty}$ defined as
\[ J_n := \big\{ a_n+1, a_n+2, a_n+3, \ldots, a_n+4C+1 \big\}. \]
Clearly, $\sum_n 1/\min J_n$ converges since $\sum_n 1/a_n$ converges.
Moreover, condition \eqref{eq:lJ1Dassume} of the lemma is fulfilled for all indices $n$ that also satisfy
\begin{equation}\label{eq:type3aux2}
a_n\geqslant 4C+1,
\end{equation}
because:
\begin{align*}
& (\min J_n)(\min J_n+1) \sum_{k=n+1}^{\infty} \frac{|J_k|-1}{(\min J_k)(\max J_k)} \\
& > a_n^2 \sum_{k=n+1}^{\infty} \frac{4C}{(a_k+1)(a_k+4C+1)}
\stackrel{\eqref{eq:type3aux2}}{\geqslant} a_n^2 \sum_{k=n+1}^{\infty} \frac{C}{a_k^2} \stackrel{\eqref{eq:type3aux1}}{\geqslant} 1,
\end{align*}
as needed.
By Lemma \ref{lm:1D} there exists a sequence of integers $(b_n)_{n=1}^{\infty}$ in $[1,4C+1]^\mathbb{N}$ such that $\sum_{n=1}^{\infty} 1/(a_n+b_n)$ is a rational number.
\end{proof}

In the particular case $a_n=2^n$, which has been mentioned by Erd\H{o}s and Graham, the proofs of Lemma \ref{lm:1D} and Theorem \ref{thm:type3} can be specialized and merged to show that
\[ \bigg\{ \sum_{n=1}^{3}\frac{1}{2^n+1} + \sum_{n=4}^{\infty}\frac{1}{2^n+b_n} \,:\, b_n\in \{1,2,\ldots,5\} \text{ for } n\geqslant 4 \bigg\} \]
equals the whole segment
\[ [0.7488145169\ldots, 0.7644997803\ldots]. \]
Consequently, there is a sequence $(b_n)_{n=1}^{\infty}$ in the set $\{1,2,\ldots,5\}$ such that
\[ \sum_{n=1}^{\infty} \frac{1}{2^n+b_n} = 0.75 \in\mathbb{Q}. \]


\section{Proof of Theorem \ref{thm:type3faster} on the existence of irrationality sequences}
\label{sec:type3construction}

Suppose that $F\colon\mathbb{N}\to(0,\infty)$ is a function satisfying the growth condition \eqref{eq:type3fa0}.
Also, let $(c_n)_{n=1}^{\infty}$ be a sequence of non-negative integers satisfying
{\allowdisplaybreaks
\begin{subequations}
\begin{align}
& \lim_{n\to\infty} c_n = \infty, \label{eq:type3fa1} \\
& \lim_{n\to\infty} \frac{c_n}{F(n)} = 0, \label{eq:type3fa2} \\
& \lim_{n\to\infty} \frac{c_n}{(F(n)/F(n-1))^2} = 0. \label{eq:type3fa3}
\end{align}
\end{subequations}
}
Such sequences $(c_n)_{n=1}^{\infty}$ clearly exist, but we prefer to work generally and postpone one explicit choice to the end of this section.
Due to \eqref{eq:type3fa0} and \eqref{eq:type3fa1}--\eqref{eq:type3fa3} there exists a positive integer $n_0$ such that
{\allowdisplaybreaks
\begin{subequations}
\begin{align}
& F(n) \geqslant 2c_n > 0, \label{eq:type3fb4} \\
& F(n)\geqslant n, \quad F(n+1)\geqslant F(n) + c_n + 1, \label{eq:type3fa4} \\
& 48 c_{n+1} + 1 < \Bigl(\frac{F(n+1)}{F(n)}\Bigr)^2 \label{eq:type3fa9}
\end{align}
\end{subequations}
}
all hold for every integer $n\geqslant n_0$.
Note that property \eqref{eq:type3fa9} can be rewritten as
\[ \frac{1}{F(n)^2} - \frac{1}{F(n+1)^2} > \frac{48c_{n+1}}{F(n+1)^2}, \]
so it also implies, by telescoping,
\begin{equation}\label{eq:type3fa5}
\frac{1}{F(n)^2} > \sum_{k=n+1}^{\infty} \frac{48c_k}{F(k)^2}
\end{equation}
for every $n\geqslant n_0$.

\begin{proposition}\label{prop:type3faster}
Consider a sequence $(a_n)_{n=1}^{\infty}$ constructed randomly by choosing
\[ a_n \in \{1,2,3,\ldots,c_n\} + \lfloor F(n)\rfloor \]
uniformly and independently for each $n\geqslant n_0$ and also set $a_n := n$ for each $1\leqslant n<n_0$.
Then, with probability $1$, $(a_n)_{n=1}^{\infty}$ is a Type 3 irrationality sequence such that $a_n \approx F(n)$.
\end{proposition}

We will first need the following auxiliary result.

\begin{lemma}\label{lm:type3fa}
Fix an integer $m\geqslant n_0$.
All real numbers of the form
\[ \sum_{n=m}^{\infty} \frac{1}{\lfloor F(n)\rfloor + d_n} \]
obtained for sequences of integers $(d_n)_{n=m}^{\infty}$ satisfying $-c_n< d_n\leqslant 2c_n$ for every $n\geqslant m$
are mutually different.
\end{lemma}

\begin{proof}[Proof of Lemma \ref{lm:type3fa}]
Suppose that $(d_k)_{k=m}^{\infty}$ and $(d'_k)_{k=m}^{\infty}$ are distinct sequences of positive integers such that $d_k,d'_k\in (-c_k, 2c_k]$ for every index $k\geqslant m$. Let $n\geqslant m$ be the smallest index at which the two sequences differ and assume, without loss of generality, that $d_n<d'_n$.
Then we can estimate
\begin{align*}
& \sum_{k=m}^{\infty} \frac{1}{\lfloor F(k)\rfloor + d_k} - \sum_{k=m}^{\infty} \frac{1}{\lfloor F(k)\rfloor + d'_k} \\
& = \frac{d'_n - d_n}{(\lfloor F(n)\rfloor + d_n)(\lfloor F(n)\rfloor + d'_n)} + \sum_{k=n+1}^{\infty} \frac{d'_k - d_k}{(\lfloor F(k)\rfloor + d_k)(\lfloor F(k)\rfloor + d'_k)} \\
& \geqslant \frac{1}{(F(n) + 2c_n)^2} - \sum_{k=n+1}^{\infty} \frac{3c_k}{(F(k) - c_k)^2}
\stackrel{\eqref{eq:type3fb4}}{\geqslant} \frac{1}{4F(n)^2} - \sum_{k=n+1}^{\infty} \frac{12c_k}{F(k)^2}
\stackrel{\eqref{eq:type3fa5}}{>} 0,
\end{align*}
so the two series that we began with have different sums.
\end{proof}

\begin{proof}[Proof of Proposition \ref{prop:type3faster}]
Let $\mathbb{P}$ be the underlying probability measure on the aforementioned collection of sequences $(a_n)_{n=1}^{\infty}$; it is the infinite product of translates of uniform probability measures $\mathbb{P}_{c_n}$ on the sets $\{1,2,3,\ldots,c_n\}$ for $n=n_0,n_0+1,\ldots$.
For a fixed $C\in\mathbb{N}$ let $m_C\geqslant n_0$ be a positive integer such that $c_n \geqslant 2C+2$ for every $n\geqslant m_C$; it exists by property \eqref{eq:type3fa1}.
Denote the event
\begin{align*}
E_C := \Big\{ (a_n)_{n=1}^{\infty} \,:\ & \text{there exists a sequence } (b_n)_{n=1}^{\infty} \text{ such that } \\
& b_n\in\mathbb{Z}, \ -C\leqslant b_n \leqslant C, \ b_n\neq -a_n \text{ for every } n\in\mathbb{N} \text{ and } \sum_{n=1}^{\infty} \frac{1}{a_n+b_n} \in \mathbb{Q} \Big\}
\end{align*}
and write it as a union $E_C = \bigcup_{q\in\mathbb{Q}} E_{C,q}$ of
\begin{align*}
E_{C,q} := \Big\{ (a_n)_{n=1}^{\infty} \,:\ & \text{there exists an integer sequence } (b_n)_{n=m_C}^{\infty} \text{ such that } \\
& -C\leqslant b_n \leqslant C \text{ for every } n\geqslant m_C \text{ and } \sum_{n=m_C}^{\infty} \frac{1}{a_n+b_n} = q \Big\}.
\end{align*}
Now we also fix $q\in\mathbb{Q}$.
From Lemma \ref{lm:type3fa} we know that there exists at most one integer sequence $(d_n)_{n=m_C}^{\infty}$ satisfying $-c_n< d_n\leqslant 2c_n$ for every $n\geqslant m_C$ and
\[ \sum_{n=m_C}^{\infty} \frac{1}{\lfloor F(n)\rfloor + d_n} = q. \]
If there are no such sequences, then clearly $E_{C,q}=\emptyset$.
If there is precisely one such sequence $(d_n)_{n=m_C}^{\infty}$, then, for every $(a_n)_{n=1}^{\infty}\in E_{C,q}$ and $(b_n)_{n=m_C}^{\infty}$ as in the definition of $E_{C,q}$, we have
\[ a_n+b_n = \lfloor F(n)\rfloor + d_n \]
for every $n\geqslant m_C$.
In the latter case,
\[ E_{C,q} = \bigcap_{n=m_C}^{\infty} \big\{ (a_n)_{n=1}^{\infty} \,:\, \lfloor F(n)\rfloor + d_n - C \leqslant a_n \leqslant \lfloor F(n)\rfloor + d_n + C \big\}, \]
which implies
\begin{align*}
\mathbb{P}(E_{C,q}) & = \lim_{N\to\infty} \prod_{n=m_C}^{N} \mathbb{P}_{c_n}\big(\big\{ i\in\{1,2,\ldots,c_n\} : d_n - C \leqslant i \leqslant d_n + C \big\}\big) \\
& \leqslant \lim_{N\to\infty} \prod_{n=m_C}^{N} \frac{2C+1}{c_n} \leqslant \lim_{N\to\infty} \prod_{n=m_C}^{N} \frac{2C+1}{2C+2} = 0.
\end{align*}
Thus, $\mathbb{P}(\bigcup_{C\in\mathbb{N}}E_{C})=0$, while the complement of $\bigcup_{C\in\mathbb{N}}E_{C}$ consists solely of Type 3 irrationality sequences $(a_n)_{n=1}^{\infty}$.
Here we also need to observe that every sequence $(a_n)_{n=1}^{\infty}$ constructed this way is strictly increasing by \eqref{eq:type3fa4} and it also satisfies $a_n\approx F(n)$ thanks to \eqref{eq:type3fa2}.
\end{proof}

\begin{remark}\label{rem:type3}
Note that the previous proof shows that almost every $(a_n)_{n=1}^{\infty}$ is an irrationality sequence in a slightly stronger sense: with the constraint $b_n\neq 0$ omitted from the definition of a Type 3 irrationality sequence.
\end{remark}

The main theorem of interest is now an easy consequence.

\begin{proof}[Proof of Theorem \ref{thm:type3faster}]
If we are only given $F$ such that \eqref{eq:type3fa0} holds, then we can choose $c_1:=0$ and
\[ c_n := \min \biggl\{\bigl\lfloor F(n)^{1/2} \bigr\rfloor,\
\biggl\lfloor\frac{F(n)}{F(n-1)}\biggr\rfloor \biggr\} \]
for $n\geqslant 2$, which clearly satisfies \eqref{eq:type3fa1}--\eqref{eq:type3fa3}.
\end{proof}


\section{Proof of Theorem \ref{thm:higherD} on higher-dimensional series}
\label{sec:proofhigher}

In order to be consistent with the next section, we prefer to consider several Ahmes series in a general ``reshuffled'' order:
\[ \sum_k \frac{1}{a_k + t_i}, \quad i=1,2,3,\ldots, \]
where $(t_i)_{i=1}^{\infty}$ is some sequence of different rational numbers.
In the proof of Theorem \ref{thm:higherD} we will simply take $t_i=i-1$, while in the proof of Theorem \ref{thm:Stolarsky} the sequence $(t_i)$ will be an enumeration of $\mathbb{Q}$. However, let us initially keep $(t_i)$ quite general and only require a ``mild'' growth condition,
\begin{equation}\label{eq:growthcond}
|t_i| \leqslant i
\end{equation}
for every $i\in\mathbb{N}$.

We need to make a change of variables $U$ that will help us measure correctly simultaneous closeness of the series' partial sums to their target values.
In other words, we want coordinates in which natural bases of neighborhoods will be formed by axes-aligned rectangular boxes of appropriate eccentricities.
To achieve that, for every $i\in\mathbb{N}$ we consider the partial fraction decomposition
\[ \frac{1}{\prod_{j=1}^{i}(x+t_j)} = \sum_{j=1}^{i} \frac{m_{i,j}}{x+t_j}, \]
which holds with the unique coefficients $m_{i,1},m_{i,2},\ldots,m_{i,i}$.
These coefficients depend on the sequence $(t_i)$, but the above identity holds for all real $x$. They are rational numbers, as a unique solution of a linear system with rational coefficients. Also, none of them is $0$, since otherwise some of the poles $x=-t_j$ of the left hand side would not be a singularity of the right hand side.
Alternatively, multiplying by $x+t_j$ and taking the limit as $x\to -t_j$ we get the explicit formula
\[ m_{i,j} = \frac{1}{\prod_{1\leqslant k\leqslant i,\, k\neq j} (t_k-t_j)}, \]
which will not be needed in the argument.
Let a linear transformation $U\colon\mathbb{R}^{\mathbb{N}}\to\mathbb{R}^{\mathbb{N}}$ be defined by
\[ U(x_i)_{i=1}^{\infty} := \biggl(\sum_{j=1}^{i} m_{i,j} x_j \biggr)_{i=1}^{\infty}. \]
For every $i\in\mathbb{N}$ define $f_i\colon\mathbb{R} \to \mathbb{R}$ as
\[ f_i(x) :=
\begin{cases}
\displaystyle\frac{1}{\prod_{j=1}^{i}(x+t_j)} & \text{ for } x\in\mathbb{R}\setminus\{-t_1,\ldots,-t_i\}, \\
0 & \text{ otherwise}
\end{cases} \]
and note that $U$ was chosen precisely to satisfy the sequence of rational function identities
\begin{equation}\label{eq:howUmaps}
U \Big( \frac{1}{x+t_1}, \frac{1}{x+t_2}, \frac{1}{x+t_3}, \ldots \Big)
= \big( f_1(x), f_2(x), f_3(x), \ldots \big).
\end{equation}
Since $U$ is clearly invertible, owing to $m_{i,i}\neq 0$ for every index $i\in\mathbb{N}$, Theorem \ref{thm:higherD} can now be restated as: for every $d\in\mathbb{N}$ there is some $\beta>1$ such that the set
\begin{align}
\Bigl\{ \Bigl( \sum_{n=1}^{\infty} f_i(a_n) \Bigr)_{i=1}^{d} \,:\
& \text{$(a_n)_{n=1}^{\infty}$ is a strictly increasing} \nonumber \\[-3mm]
& \text{sequence in $\mathbb{N}$ such that } \lim_{n\to\infty} a_n^{1/\beta^n} = \infty \Bigr\} \label{eq:strongerhd}
\end{align}
has a non-empty interior in $\mathbb{R}^d$.

We begin with the following two auxiliary results.
The first one asserts that $f_i$ behaves locally like a linear function of slope $-i/N^{i+1}$ near a given large number $N$.

\begin{lemma}\label{lm:infDest}
For every $i\in\mathbb{N}$ there exists a constant $C_i\in(0,\infty)$ such that the inequality
\begin{equation}\label{eq:infDsillyineq}
\Big| f_i(N) - f_i(N+n) - \frac{i n}{N^{i+1}} \Big| \leqslant \frac{C_i n^2}{N^{i+2}}
\end{equation}
holds for all $N\in\mathbb{N}$ and $n\in\mathbb{Z}$ such that $|n|\leqslant N/4i$.
\end{lemma}

\begin{proof}[Proof of Lemma \ref{lm:infDest}]
Inequality \eqref{eq:infDsillyineq} is trivial for $n=0$, so assume that $n\neq 0$, which also implies $N\geqslant 4i$.
Consider the expression inside the absolute values on the left hand side of \eqref{eq:infDsillyineq}.
Recalling \eqref{eq:growthcond}, we see that its denominator is
\[ N^{i+1} \prod_{j=1}^{i} (N+t_j)(N+n+t_j) \geqslant 2^{-2i} N^{3i+1}. \]
On the other hand, its numerator
\[ N^{i+1} \prod_{j=1}^{i} (N+n+t_j) - N^{i+1} \prod_{j=1}^{i} (N+t_j) - i n \prod_{j=1}^{i} (N+t_j)(N+n+t_j) \]
is, in the absolute value, at most
\[
N^{i+1} 2^i N^{i-2} (|n|+i)^2
+ N^{i+1} 2^i N^{i-2} i^2
+ i n 2^{2i} N^{2i-1} (|n|+i)
\leqslant 2^{8i} n^2 N^{2i-1},
\]
due to the cancellation of the coefficients next to $N^{2i+1}$ and $N^{2i}$, when viewing it as a polynomial in $N$.
After division we see that \eqref{eq:infDsillyineq} is satisfied with $C_i = 2^{10i}$.
\end{proof}

The second auxiliary result will be crucial for the approximation of certain points in $\mathbb{R}^d$ by the Ahmes series' terms.

\begin{lemma}\label{lm:quantest}
For every $d\in\mathbb{N}$ there exist constants $0<\varepsilon_d\leqslant 1\leqslant D_d<\infty$ such that the following holds for all $N\in\mathbb{N}$ and $M\in\mathbb{Z}$ satisfying $0\leqslant M\leqslant N/4d$.
Define a point $s\in\mathbb{R}^d$ and a set $S\subset\mathbb{R}^d$ as
\begin{align*}
s & := \Big( \sum_{j=1}^{d} f_i(j N) \Big)_{i=1}^{d}, \\
S & := \Bigl\{ \Big( \sum_{j=1}^{d} f_i(j N + n_j) \Big)_{i=1}^{d} \,:\, n_1,\ldots,n_d\in[-M,M]\cap\mathbb{Z} \Bigr\}.
\end{align*}
Then
\begin{equation}\label{eq:mainsubset}
s + \varepsilon_d \prod_{i=1}^{d} \Bigl[ -\frac{M}{N^{i+1}}, \frac{M}{N^{i+1}} \Bigr]
\subseteq S + D_d \prod_{i=1}^{d} \biggl[ -\Big( \frac{1}{N^{i+1}} +\frac{M^2}{N^{i+2}} \Big), \Big( \frac{1}{N^{i+1}} +\frac{M^2}{N^{i+2}} \Big) \biggr].
\end{equation}
\end{lemma}

If $M$ is taken to be roughly equal to $N^{1/2}$, then the lemma can be interpreted as follows. Every point in $\mathbb{R}^d$ that differs from $s$ in the $i$-th coordinate by $O(N^{-i-1/2})$ for $i=1,2,\ldots,d$ can be approximated by a point from $S$ up to an error which is $O(N^{-i-1})$ in the $i$-th coordinate for each $i$. This ``gain'' of $N^{-1/2}$ will allow us to construct successive approximations of points from a whole rectangular box by terms of a subseries of $(\sum_n f_i(n))_{i=1}^{d}$.

\begin{proof}[Proof of Lemma \ref{lm:quantest}]
For $n_1,\ldots,n_d\in\mathbb{Z}$ denote
\[ p(n_1,\ldots,n_d) := \Big( \frac{i}{N^{i+1}} \sum_{j=1}^{d} \frac{n_j}{j^{i+1}} \Big)_{i=1}^{d} \in \mathbb{R}^d \]
and
\[ \Delta(n_1,\ldots,n_d) := \biggl( \sum_{j=1}^{d} \Bigl( f_i(j N)-f_i(j N + n_j)-\frac{i n_j}{j^{i+1} N^{i+1}} \Bigr) \biggr)_{i=1}^{d} \in \mathbb{R}^d. \]
The set $S$ can now be written as
\[ S = \Bigl\{ s - p(n_1,\ldots,n_d) - \Delta(n_1,\ldots,n_d) \,:\, n_1,\ldots,n_d\in[-M,M]\cap\mathbb{Z} \Bigr\}, \]
while inequality \eqref{eq:infDsillyineq} and $\sum_{j=1}^{d}j^{-i-2}<2$ guarantee that
\begin{equation}\label{eq:deltaQ}
\Delta(n_1,\ldots,n_d) \in \prod_{i=1}^{d} \Bigl[ -\frac{2C_i M^2}{N^{i+2}}, \frac{2C_i M^2}{N^{i+2}} \Bigr]
\end{equation}
for every integers $n_1,\ldots,n_d\in[-M,M]$.

Let $V,D\colon\mathbb{R}^d\to\mathbb{R}^d$ be linear transformations given by their matrices
\[ V = \big(j^{-i-1}\big)_{\substack{1\leqslant i\leqslant d\\ 1\leqslant j\leqslant d}}
= \begin{pmatrix} 1^{-2} & 2^{-2} & \cdots & d^{-2} \\
1^{-3} & 2^{-3} & \cdots & d^{-3} \\
\vdots & \vdots & \cdots & \vdots \\
1^{-d-1} & 2^{-d-1} & \cdots & d^{-d-1}
\end{pmatrix} \]
and
\[ D = \mathop{\textup{diag}}\big( 1^{d+1}, 2^{d+1}, \ldots, d^{d+1} \big) \]
in the standard basis of $\mathbb{R}^d$.
Note that
\[ VD = \begin{pmatrix} 1^{d-1} & 2^{d-1} & \cdots & d^{d-1} \\
\vdots & \vdots & \cdots & \vdots \\
1^{1} & 2^{1} & \cdots & d^{1} \\
1^{0} & 2^{0} & \cdots & d^{0}
\end{pmatrix} \]
has integer coefficients and the formula for the Vandermonde determinant gives
\[ v_d := |\det(VD)| = \prod_{1\leqslant i<j\leqslant d} (j-i) \in \mathbb{N}. \]
Then $v_d (VD)^{-1}$ exists and it also has integer coefficients, so it maps $v_d (VD)^{-1} \mathbb{Z}^d \subseteq \mathbb{Z}^d$.
Consequently,
\[ V\mathbb{Z}^d \supseteq VD\mathbb{Z}^d \supseteq v_d (VD) (VD)^{-1} \mathbb{Z}^d = v_d \mathbb{Z}^d, \]
i.e., $V\mathbb{Z}^d$ contains the integer sub-lattice $v_d \mathbb{Z}^d$.
If we also denote
\[ \varepsilon_d := \inf_{x\in\mathbb{R}^d,\,x\neq\mathbf{0}} \frac{|Vx|_{\infty}}{|x|_{\infty}} =  >0, \]
then
\[ V \big( [-M,M]^d \cap \mathbb{Z}^d \big) \supseteq[-\varepsilon_d M,\varepsilon_d M]^d \cap v_d\mathbb{Z}^d, \]
so, recalling the definitions of $p$ and $V$, we obtain
\begin{equation}\label{eq:pkcontains}
\Bigl\{ p(n_1,\ldots,n_d) \,:\, n_1,\ldots,n_d\in[-M,M]\cap\mathbb{Z} \Bigr\}
\supseteq \prod_{i=1}^{d} \biggl( \Bigl[-\frac{\varepsilon_d i M}{N^{i+1}},\frac{\varepsilon_d i M}{N^{i+1}}\Bigr] \cap \frac{i v_d}{N^{i+1}}\mathbb{Z} \biggr).
\end{equation}

Now we can finalize the proof.
For every point $x$ from the rectangular box on the left hand side of \eqref{eq:mainsubset} we have
\[ s-x\in \prod_{i=1}^{d} \Bigl[-\frac{\varepsilon_d i M}{N^{i+1}},\frac{\varepsilon_d i M}{N^{i+1}}\Bigr], \]
so by \eqref{eq:pkcontains} we can find $n_1,\ldots,n_d\in[-M,M]\cap\mathbb{Z}$ such that
\begin{equation}\label{eq:pkcontains2}
p(n_1,\ldots,n_d) - (s-x) \in \prod_{i=1}^{d} \Bigl[-\frac{i v_d}{N^{i+1}},\frac{i v_d}{N^{i+1}}\Bigr].
\end{equation}
If we define
\[ y := s - p(n_1,\ldots,n_d) - \Delta(n_1,\ldots,n_d) \in S, \]
then, by \eqref{eq:pkcontains2} and \eqref{eq:deltaQ},
\begin{align*}
x - y & = \big(x-s+p(n_1,\ldots,n_d)\big) + \Delta(n_1,\ldots,n_d) \\
& \in \prod_{i=1}^{d} \Bigl[ -\frac{i v_d}{N^{i+1}} -\frac{2C_i M^2}{N^{i+2}}, \frac{i v_d}{N^{i+1}} +\frac{2C_i M^2}{N^{i+2}} \Bigr],
\end{align*}
so \eqref{eq:mainsubset} follows with
\[  D_d := \max\{ d v_d, 2C_1, \ldots, 2C_d \}. \qedhere \]
\end{proof}

Concrete values of the constants $\varepsilon_d$ and $D_d$ are certainly not important, but a closer inspection of the arguments can easily replace them with
\[ \varepsilon_d = d^{-10d^2},\quad D_d = (4d)^{10d^2}. \]

After all this preliminary material, we are ready for the actual proof of Theorem \ref{thm:higherD}.
In some sense we are adapting the proof of Lemma \ref{lm:1D} to several dimensions.
As we have already announced, in the following proof we take $t_i=i-1$, but it works equally well for any concrete choice of $(t_i)$.

\begin{proof}[Proof of Theorem \ref{thm:higherD}]
Fix a positive integer $d$ and let $\varepsilon_d$ and $D_d$ be the constants from Lemma \ref{lm:quantest}.
We can now prove that \eqref{eq:strongerhd} has a non-empty interior for any $\beta$ that satisfies
\begin{equation}\label{eq:thebeta}
1 < \beta < \Big(\frac{2d+2}{2d+1}\Big)^{1/d}.
\end{equation}
After fixing $\beta$, we further select any number $\alpha$ such that
\begin{equation}\label{eq:thealpha}
\beta^d < \alpha < \frac{2d+2}{2d+1}
\end{equation}
and define a sequence of positive integers $(N_k)_{k=1}^{\infty}$ by
\[ N_k := \begin{cases}
(2d+1)^k & \text{for } 1\leqslant k\leqslant k_0, \\
\big\lfloor 2^{\alpha^k} \big\rfloor & \text{for } k>k_0,
\end{cases} \]
with $k_0\in\mathbb{N}$ chosen sufficiently large that
\begin{equation}\label{eq:whatisn1}
N_{k+1} \geqslant (2d+1) N_k
\end{equation}
holds for every $k\in\mathbb{N}$.
Note that the choice \eqref{eq:thealpha} of $\alpha$ guarantees
\begin{equation}\label{eq:whoan}
\lim_{k\to\infty} \frac{N_{k+1}}{N_k^{(2d+2)/(2d+1)}} = 0.
\end{equation}
Also define
\begin{equation}\label{eq:whatism}
M_k := \min\Bigl\{ \Big\lfloor\frac{N_k}{4d}\Big\rfloor , \big\lfloor N_k^{1/2}\big\rfloor \Bigr\},
\end{equation}
so that, in particular,
\begin{equation}\label{eq:whatism9}
M_k\approx N_k^{1/2}.
\end{equation}

The principal idea is to consider the collection of sequences
\[ \mathcal{A} := \big\{ (a_n)_{n=1}^{\infty} \in {\mathbb{Z}}^{\mathbb{N}} \,:\, a_{(k-1)d+j} \in [j N_k-M_k,j N_k+M_k] \text{ for } k\in\mathbb{N} \text{ and } j\in\{1,2,\ldots,d\} \big\} \]
and prove that the set
\begin{equation}\label{eq:fullsersums}
\Bigl\{ \Bigl( \sum_{n=1}^{\infty} f_i(a_n) \Bigr)_{i=1}^{d} \,:\, (a_n)_{n=1}^{\infty} \in\mathcal{A} \Bigr\} \subset \mathbb{R}^d
\end{equation}
has a non-empty interior.
The assumption \eqref{eq:whatisn1} and the definition \eqref{eq:whatism} guarantee that every sequence from $\mathcal{A}$ is positive and strictly increasing.
Moreover, for every sequence $(a_n)_{n=1}^{\infty}$ from $\mathcal{A}$ and every sufficiently large index $n=(k-1)d+j$ for some $k\in\mathbb{N}$ and $j\in\{1,2,\ldots,d\}$, we have
\[ a_n \geqslant N_k - M_k \stackrel{\eqref{eq:whatism}}{>} \frac{1}{2} N_k \geqslant \frac{1}{4} 2^{\alpha^{n/d}}. \]
Thus, $\lim_{n\to\infty} a_n^{1/\beta^n} = \infty$, thanks to the choice \eqref{eq:thealpha} of $\alpha$ again.

For every $k\in\mathbb{N}$ denote
\begin{align*}
s_k & := \Big( \sum_{j=1}^{d} f_i(j N_k) \Big)_{i=1}^{d} \in \mathbb{R}^d, \\
S_k & := \Bigl\{ \Big( \sum_{j=1}^{d} f_i(j N_k + n_j) \Big)_{i=1}^{d} \,:\, n_1,\ldots,n_d\in[-M_k,M_k]\cap\mathbb{Z} \Bigr\} \subset \mathbb{R}^d,
\end{align*}
and
\[ R_k := \prod_{i=1}^{d} \Bigl[ -\frac{\varepsilon_d M_k}{N_k^{i+1}}, \frac{\varepsilon_d M_k}{N_k^{i+1}} \Bigr]. \]
Lemma \ref{lm:quantest} will apply and give
\begin{equation}\label{eq:inductivesubset}
s_k + R_k \subseteq S_k + R_{k+1}
\end{equation}
as soon as we have
\begin{equation}\label{eq:hdmaincond}
\frac{D_d}{N_k^{i+1}} + \frac{D_d M_k^2}{N_k^{i+2}} \leqslant \frac{\varepsilon_d M_{k+1}}{N_{k+1}^{i+1}},
\end{equation}
but this is satisfied for sufficiently large indices $k\in\mathbb{N}$ and every $1\leqslant i\leqslant d$.
Namely, after recalling \eqref{eq:whatism9}, the quotient of the left and the right hand sides of \eqref{eq:hdmaincond} is asymptotically, as $k\to\infty$ for a fixed $i$, equal to
\[ \frac{2D_d}{\varepsilon_d} \frac{N_{k+1}^{i+1/2}}{N_k^{i+1}} \stackrel{k\to\infty}{\longrightarrow} 0, \]
where we used \eqref{eq:whoan}.

An easy consequence of the rapid growth of $(N_k)$ and the definition of $f_i$ is the convergence in $\mathbb{R}^d$ of any series $\sum_{k=1}^{\infty} y_k$ with terms $y_k\in S_k$,
and observe that \eqref{eq:fullsersums} is precisely equal to the set of all their sums.
Let $m\in\mathbb{N}$ be sufficiently large so that \eqref{eq:hdmaincond}, and thus also \eqref{eq:inductivesubset}, is fulfilled for every $k\geqslant m$.
We claim that the set of series' tail sums
\begin{equation}\label{eq:hdtail}
\Big\{ \sum_{k=m}^{\infty} y_k \,:\, y_k\in S_k \text{ for every } k\geqslant m \Big\}
\end{equation}
contains the whole rectangular box $\sum_{k=m}^{\infty} s_k + R_m$.
Indeed, for any $x$ from this box one repeatedly uses \eqref{eq:inductivesubset} to inductively construct a sequence of points $y_k\in S_k$ for $k=m,m+1,m+2,\ldots$ such that
\[ x \in \sum_{k=m}^{n-1} y_k + \sum_{k=n}^{\infty} s_k + R_n \]
for every integer $n\geqslant m$. Because the diameters of $R_n$ converge to $0$, we can pass to the limit as $n\to\infty$ and obtain
\[ x = \sum_{k=m}^{\infty} y_k, \]
which is an element of \eqref{eq:hdtail}.
As a consequence, the set of full series' sums \eqref{eq:fullsersums} also contains a non-degenerate rectangular box and thus has a non-empty interior.
\end{proof}


\section{Proof of Theorem \ref{thm:Stolarsky} on Stolarsky's problem}
\label{sec:proofStolarsky}

Here we assume that $(t_i)_{i=1}^{\infty}$ enumerates the set of rational numbers and satisfies \eqref{eq:growthcond}, so that Lemmas \ref{lm:infDest} and \ref{lm:quantest} can be used again.
Recall the transformation $U$ from the previous section; it clearly satisfies
\[ U \mathbb{Q}^{\mathbb{N}} = \mathbb{Q}^{\mathbb{N}} . \]
Also recall the sequence of functions $f_1,f_2,\ldots$ defined there. Because of \eqref{eq:howUmaps}, Theorem \ref{thm:Stolarsky} will follow from the fact that there exists a strictly increasing sequence of positive integers $(a_n)_{n=1}^{\infty}$ such that $\sum_n 1/a_n < \infty$ and
\[ \Bigl( \sum_{n=1}^{\infty} f_i(a_n) \Bigr)_{i=1}^{\infty} \in \mathbb{Q}^{\mathbb{N}}. \]
Lemma \ref{lm:quantest} will, once again, be the main ingredient in the proof.

\begin{proof}[Proof of Theorem \ref{thm:Stolarsky}]
The sequence $(N_k)_{k=1}^{\infty}$ is now defined as
\[ N_k := \begin{cases}
(2k-1)! & \text{for } 1\leqslant k\leqslant k_0, \\
\big\lfloor 2^{2^{\sqrt{k}}} \big\rfloor & \text{for } k>k_0,
\end{cases} \]
with $k_0\in\mathbb{N}$ chosen sufficiently large that
\begin{equation}\label{eq:iwhatisn1}
N_{k+1} \geqslant (2k+1) N_k
\end{equation}
for every $k\in\mathbb{N}$.
This time we have
\begin{equation}\label{eq:whattnn}
\lim_{k\to\infty} \frac{N_{k+1}}{N_k^{\theta}} = 0
\end{equation}
for any given $\theta\in(1,\infty)$, arbitrarily close to $1$.
Let $(M_k)_{k=1}^{\infty}$ be given by
\[ M_k := \min\Bigl\{ \Big\lfloor\frac{N_k}{4k}\Big\rfloor , \big\lfloor N_k^{1/2}\big\rfloor \Bigr\}, \]
so that \eqref{eq:whatism9} holds again.

First, construct a sequence of positive integers
\[ m(1) < m(2) < m(3) < \ldots \]
inductively as follows.
For each $d\in\mathbb{N}$ let $m(d)$ be the least positive integer larger than $m(d-1)$ such that \eqref{eq:hdmaincond} is valid for every $1\leqslant i\leqslant d$ and $k\geqslant m(d)$. (Here we set $m(0):=0$.) Such a number exists because of \eqref{eq:whattnn} by the same argument given in the proof of Theorem \ref{thm:higherD}.
Also, for any integer $k\geqslant m(1)$ let $d(k)$ be the unique positive integer such that
\[ m(d(k))\leqslant k<m(d(k)+1). \]
Since $m(i)\geqslant i$, we also have
\begin{equation}\label{eq:lessthani}
d(k) \leqslant m(d(k))\leqslant k.
\end{equation}
Finally, denote
\[ \delta_{i,k} := \frac{\varepsilon_{d(k)} M_k}{N_k^{i+1}} \]
for all positive integers $i$ and $k$ such that $i\leqslant d(k)$.

We describe an algorithm for generating both a sequence of rational numbers $(x_i)_{i=1}^{\infty}$ and a strictly increasing sequence of positive integers $(a_n)_{n=1}^{\infty}$ such that
\begin{equation}\label{eq:Stolprop}
\sum_{n=1}^{\infty} f_i(a_n)=x_i
\end{equation}
for every $i\in\mathbb{N}$.
Denote
\[ n(k) := \sum_{l=m(1)}^{k-1} d(l) \]
for every integer $k\geqslant m(1)$, noting that $n(m(1))=0$.
The sequence $(a_n)$ will, more precisely, be such that
\begin{equation}\label{eq:infinduction}
x_i \in \sum_{l=m(1)}^{k-1} \sum_{j=1}^{d(l)} f_i(a_{n(l)+j}) + \sum_{l=k}^{\infty} \sum_{j=1}^{d(l)} f_i(j N_l) + \big[-\delta_{i,k},\delta_{i,k}\big]
\end{equation}
holds for every $i\in\mathbb{N}$ and every integer $k\geqslant m(i)$.
Afterwards, \eqref{eq:Stolprop} will be obtained simply by letting $k\to\infty$ for each fixed $i\in\mathbb{N}$ and observing $\lim_{k\to\infty}\delta_{i,k}=0$.

The steps of the algorithm are indexed by the integers $k\geqslant m(1)$.
Before the $k$-th step, the terms $(a_n)_{n=1}^{n(k)}$ are already defined in a way that \eqref{eq:infinduction} holds for every $i\in\mathbb{N}$ such that $m(i)<k$.
If there exist (a unique) $i\in\mathbb{N}$ such that $m(i)=k$, then we pick a rational number $x_i$ from the interval
\[ \sum_{l=m(1)}^{m(i)-1} \sum_{j=1}^{d(l)} f_i(a_{n(l)+j}) + \sum_{l=m(i)}^{\infty} \sum_{j=1}^{d(l)} f_i(j N_l) + \big[-\delta_{i,m(i)},\delta_{i,m(i)}\big]. \]
Note that this trivially fulfills \eqref{eq:infinduction} also when $m(i)=k$.
Now, Lemma \ref{lm:quantest} applies with $d=d(k)$, $M=M_k$, $N=N_k$ and it gives
\begin{align*}
& \Big( \sum_{j=1}^{d(k)} f_i(j N_k) \Big)_{i=1}^{d(k)} + \prod_{i=1}^{d(k)} \big[-\delta_{i,k},\delta_{i,k}\big] \\
& \subseteq \Bigl\{ \Big( \sum_{j=1}^{d(k)} f_i(j N_k + n_j) \Big)_{i=1}^{d(k)} \,:\, n_1,\ldots,n_{d(k)}\in[-M_k,M_k]\cap\mathbb{Z} \Bigr\} + \prod_{i=1}^{d(k)} \big[-\delta_{i,k+1},\delta_{i,k+1}\big],
\end{align*}
since we know that
\[ \frac{D_{d(k)}}{N_{k}^{i+1}} + \frac{D_{d(k)}M_k^2}{N_{k}^{i+1}} \leqslant \delta_{i,k+1} \]
holds for every $1\leqslant i\leqslant d(k)$ by the definitions of $d(k)$ and $\delta_{i,k}$.
In the $k$-th step of the algorithm, we combine this with \eqref{eq:infinduction} to find integers $n_1,\ldots,n_{d(k)}\in[-M_k,M_k]$ such that
\[ x_i \in \sum_{l=m(1)}^{k-1} \sum_{j=1}^{d(l)} f_i(a_{n(l)+j}) + \sum_{j=1}^{d(k)} f_i(j N_k + n_j) + \sum_{l=k+1}^{\infty} \sum_{j=1}^{d(l)} f_i(j N_l) + \big[-\delta_{i,k+1},\delta_{i,k+1}\big] \]
whenever $i\leqslant d(k)$, i.e., $m(i)\leqslant k$.
It remains to append $(a_n)_{n=1}^{n(k)}$ with $d(k)$ new terms,
\[ a_{n(k)+j} := j N_k + n_j, \quad 1\leqslant j\leqslant d(k), \]
and observe that the previous display can be rewritten as
\[ x_i \in \sum_{l=m(1)}^{k} \sum_{j=1}^{d(l)} f_i(a_{n(l)+j}) + \sum_{l=k+1}^{\infty} \sum_{j=1}^{d(l)} f_i(j N_l) + \big[-\delta_{i,k+1},\delta_{i,k+1}\big] \]
and that it is valid when $m(i)<k+1$.
This is precisely \eqref{eq:infinduction} with $k$ replaced by $k+1$, which finalizes the $k$-th step of our algorithm.

Each sequence $(a_n)$ constructed in this way is positive and strictly increasing because of \eqref{eq:iwhatisn1} and \eqref{eq:lessthani}.
\end{proof}


\section*{Acknowledgments}
The authors are grateful to Thomas Bloom for founding and managing the website \cite{EP}, and for discussing several of the problems considered in this paper.
V.K. would also like to thank Ilijas Farah for a useful discussion of the existing literature.
V.K. is supported in part by the Croatian Science Foundation under the project HRZZ-IP-2022-10-5116 (FANAP). T.T. is supported by NSF grant DMS-2347850.


\bibliography{AhmesSeries}
\bibliographystyle{plainurl}

\end{document}